\documentclass[11pt]{article}
\usepackage{latexsym,amsmath,amssymb, amsthm}

\usepackage[applemac]{inputenc}
\usepackage{indentfirst}
\usepackage[mathscr]{eucal}
\usepackage{amssymb,amsmath,amsfonts}
\usepackage{fancybox,fancyhdr}
\usepackage{graphicx}
\usepackage{url}
\usepackage{graphicx}
\usepackage{epstopdf}
\usepackage{wrapfig}
\usepackage{float}
\usepackage[usenames]{color}
\usepackage{colortbl}
\usepackage[active]{srcltx}
\newif\ifdviwin

\dviwintrue

\def\E{\mathbb{E}}

\def\R{\mathbb{R}}

\def\D{\mathbb{D}}
\def\H{\mathbb{H}}
\def\S{\mathbb{S}}

 \newtheorem{defi}{Definition}
  \newtheorem{teo}{Theorem}
 
 \newtheorem{cor}{Corollary}
 \newtheorem{lem}{Lemma}
  \newtheorem{glemma}{Generalized lemma}
 \newtheorem{remark}{Remark}
    \newtheorem{asser}{{Assertion}}[section]
 \newenvironment{proof1}{\rm \trivlist \item[\hskip \labelsep{\it
      Proof of Assertion \ref{asser1}}.]}{\par\nopagebreak \hfill $\square$ \endtrivlist}
      \newenvironment{proof2}{\rm \trivlist \item[\hskip \labelsep{\it
      Proof of Lemma \ref{l41}}.]}{\nopagebreak \hfill $\Box$ \endtrivlist}
          \newenvironment{proof3}{\rm \trivlist \item[\hskip \labelsep{\it
      Proof of Theorem \ref{main}}.]}{\nopagebreak \hfill $\square$ \endtrivlist}
          \newenvironment{proof4}{\rm \trivlist \item[\hskip \labelsep{\it
      Proof of Theorem \ref{ths41}}.]}{\nopagebreak \hfill $\square$ \endtrivlist}
          \newenvironment{proof5}{\rm \trivlist \item[\hskip \labelsep{\it
      Proof of Theorem \ref{ths42}}.]}{\nopagebreak \hfill $\square$ \endtrivlist}
\numberwithin{equation}{section}
\begin{document}
\renewcommand{\thefootnote}{}
\footnotetext{Research partially supported by Ministerio de Educaci\'on Grant No. MTM2010-19821, Junta de Andaluc\'\i a Grants No. FQM325, No. P06-FQM-01642. }
\title{Complete Surfaces with Ends of Non Positive Curvature}
\author{Jos\'e A. G\'alvez, Antonio Mart\'{\i}nez and Jos\'e L. Teruel}
\maketitle
\begin{abstract} 
In this paper we extend  Efimov's Theorem by proving that any complete surface in $\R^3$ with  Gauss curvature  bounded above by a negative constant outside a compact set has  finite total curvature, finite area and is properly immersed. Moreover, its ends must be  asymptotic to half-lines.  We also give a partial solution to Milnor's conjecture by studying  isometric immersions in a space form of complete surfaces which satisfy that outside a compact set they have  non positive Gauss curvature  and the square of a principal curvature function is bounded from below by a positive constant.
\end{abstract}
\section{Introduction}
An important part in the study of complete  surfaces  of non positive curvature in $\R^3$
has been  directed at nonexistence of isometric immersions. The investigation of the isometric immersion of metrics with negative curvature goes back to Hilbert. He proved  in 1901, see \cite{H,H2}, that the full hyperbolic plane can not be isometrically immersed in $\R^3$. This means, it is impossible to extend a regular piece of a surface of constant negative curvature without the appearance of singularities. Hilbert's theorem has attracted the attention of many mathematician and an important amount of work has been done in this direction (see for example \cite{A, B, B2,E,HW,Ho,O,St,St2}). 

A next natural step, conjectured by many geometers  \cite{CV} was to extend such a result to complete surfaces which Gauss curvature is bounded above by a negative constant. The final solution to this problem was obtained by Efimov in 1963 more than sixty years later, see \cite{E,KM}. Efimov proved that no surface can be 
${\cal C}^2$-immersed in the  Euclidean 3-space so as to be complete in the induced Riemannian metric, with Gauss curvature $K \leq \text{\rm const} < 0$.    In the  following years  Efimov extended this result in several ways, see \cite{E2}.

Although Efimov's proof is very delicate,  it is ingenious and does not depend upon sophisticated or modern techniques.  In fact, it is  derived from a general result  about ${\cal C}^1$-immersions between planar domains, see \cite[Lemma 1]{E} and \cite[Main Lemma]{KM}.  Here, we will show that the same method can be applied to study complete surfaces which  Gauss  curvature is negative  and bounded away from zero outside a compact subset.  We shall prove, see Theorem  \ref{ths42}, that: 
 \begin{quote}
{\sl Any complete ${\cal C}^2$-immersed surface in $\R^3$ with Gauss curvature satisfying $K\leq \text{\rm const} <0$ outside a compact subset,  has finite total curvature, finite area and  is properly immersed. Moreover, its ends are   asymptotic to half-lines, see Figure \ref{f1a}.}
\end{quote}
\begin{figure}[H]
\center{\includegraphics[width=0.4\linewidth]{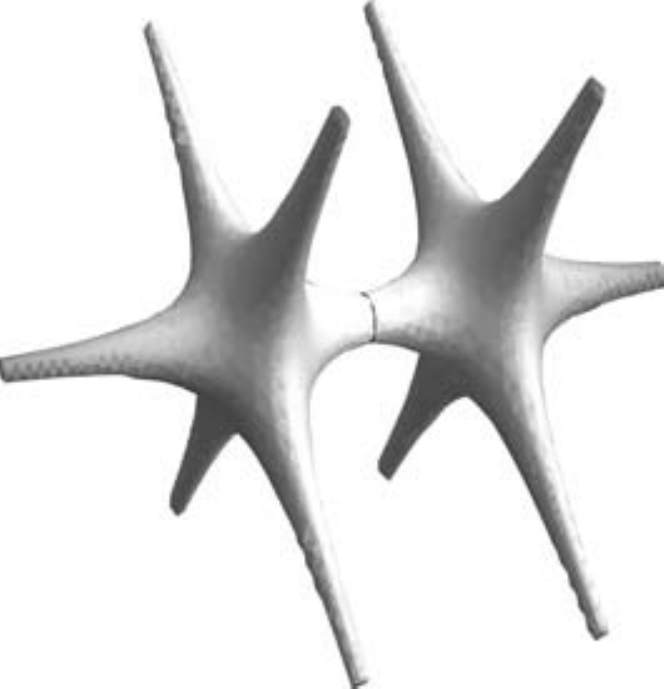}\hspace{1cm}
\includegraphics[width=0.38\linewidth]{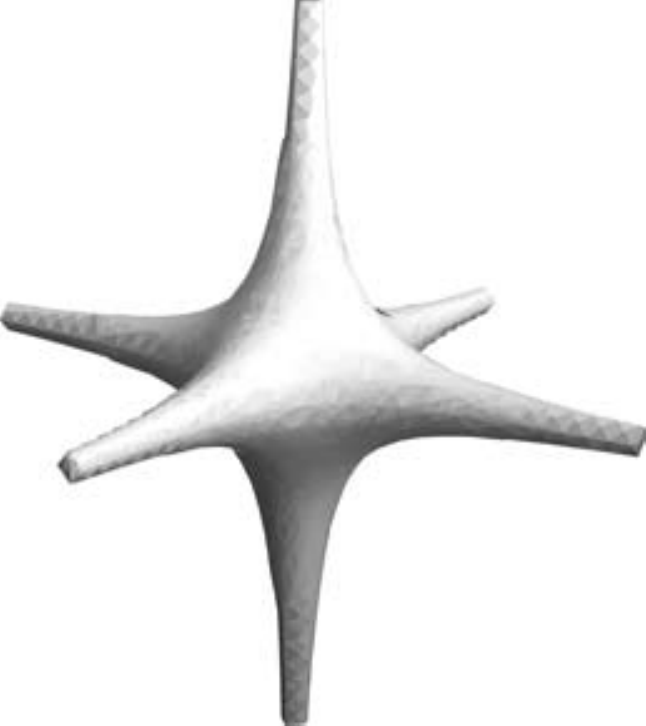}}
\caption{Complete surfaces with negatively curved ends.}
\label{f1a}
\end{figure}

A more general result than the above mentioned  is indeed proven, see Theorem \ref{main} and Theorem \ref{ths41}.

Despite the considerable progress in the understanding of negatively curved surfaces, important questions suggested by Hilbert's Theorem remain unanswered to this day. Among the most interesting  open problems  we mention the following conjecture due  to John Milnor, (see \cite{KO}),
 \begin{quote}
{\sc Milnor's Conjecture}.  {\sl Assume $\psi: \Sigma \longrightarrow \R^3$ is a complete, umbilic free immersion so that their  principal curvatures $k_1$ and $k_2$ satisfy 
$$ k_1^2 + k_2^2\geq \text{\rm const}>0.$$ Then either the Gauss curvature $K$ of $\psi$ changes sign  or else $K\equiv 0$}.
\end{quote}

In this paper,  we also aim  at taking  a small step toward the  solution of the above mentioned conjecture and its extension to other non-Euclidean space forms. First, we prove (Theorem \ref{th1}):
 \begin{quote}
 {\sl Any complete surface of non positive Gauss curvature  isometrically immersed in $\R^3$ with one of its principal curvature functions $k_i$ satisfying
 $$ k_i^2 \geq \text{\rm const} >0,$$
must be  a generalized cylinder.}
\end{quote}

An interesting  consequence  of this result is, for instance, that  generalized cylinders are the only complete surfaces  with non positive Gauss curvature, isometrically immersed in $\R^3$ with mean curvature bounded away from zero (Corollary \ref{cor1}).  As another consequence,  we  generalize a result of Klotz and Osserman in \cite{KO} by showing  that any complete special Weingarten surface in $\R^3$ on which  the Gauss curvature does not change of sign is either a round sphere or a right circular cylinder. This gives a positive answer to a question raised by Sa Earp and Toubiana in \cite{ST}.

Although an analogous to Efimov's Theorem in non-Euclidean space forms remains as an open problem to this day, a partial solution was obtained by Schlenker in \cite{S}. Using a slightly different approach, the above mentioned  results can be extended to other ambient spaces by application of  the abstract theory of Codazzi pairs,  that is,  pairs $(I,II)$ of real quadratic forms on an abstract surface, where $I$ is a Riemann metric and $II$ satisfies the Codazzi-Mainardi equations of the classical surface theory with respect to the metric $I$. 

An abstract result  about  Codazzi  pairs, see Theorem \ref{at1}, let us to prove some   consequences about  immersions in the hyperbolic 3-space $\H^3$  of curvature $-1$ and in the sphere $\S^3$ of curvature $1$. Actually, we prove (see Corollaries \ref{ac1} and \ref{ac2}):
 \begin{quote} 
 {\sl No surface can be 
immersed in  $\H^3$ (resp. $\S^3$) if it is complete in the induced Riemannian metric, with Gauss curvature $K \leq -1$ (resp. $K \leq const <0$) and one of its principal curvature functions $k_i$ satisfying $$k_i^2\geq \epsilon >0,$$ for some positive constant $\epsilon$.}
\end{quote}

About the geometry of surfaces with ends of non positive curvature in non Euclidean space forms we can prove, see Corollaries \ref{c4} and \ref{c5}:
 \begin{quote} 
 {\sl Consider a   complete immersion in  $\H^3$ (resp. $\S^3$) satisfying that  outside a compact subset,
 \begin{itemize} 
 \item the  Gauss curvature $K \leq -1$ (resp. $K \leq const <0$) and 
 \item  $k_i^2\geq \epsilon >0$, $\epsilon\in \R,$ where $k_i$ is one of its principal curvature functions.
\end{itemize}
Then it   has finite total curvature and, in particular, it  has finite topology and finite area.}
\end{quote}
 
\section{A step in the solution of  Milnor's conjecture}
In this section we prove  a partial solution to Milnor's conjecture and extend  the result of Klotz and Osserman in \cite{KO} answering  the question raised by Sa Earp and Toubiana in \cite{ST}.

Along  all the section we shall always assume that the differentiability used is  ${\cal C}^\infty$ but the differentiability requirements are actually much lower. Indeed, for most of the cases ${\cal C}^3$-differentiability  will be enough.

We  also suppose that $\Sigma$ is an oriented surface (otherwise we would work with its oriented two-sheeted covering). 

First, we will show the following  result holds:
\begin{teo}\label{th1}
Let  $\psi: \Sigma \longrightarrow \R^3$  be a complete immersed surface of  non positive curvature.  If one of its principal curvatures $k_i$ satisfies
$$ k_i^2 \geq const >0, $$ 
then $\psi(\Sigma)$ is a generalized cylinder in $\R^3$.
\end{teo}
\begin{proof}
Let us denote by $k_1$ and $k_2$ the principal curvatures of $\psi$. Up to a change of orientation, we can assume that
\begin{equation}\label{condition1}k_2\geq \epsilon >\frac{\epsilon}{2}>0 > k_1, \qquad \text{\rm for some positive constant } \epsilon.\end{equation} 

Consider $\psi_\epsilon$ the parallel map of $\psi$ to a distance  $2/\epsilon$, that is, 
$$\psi_\epsilon := \psi + \frac{2}{\epsilon} N,$$
where  $N:\Sigma \longrightarrow \S^2$ is the Gauss map of $\psi$. Then, it is not difficult to check that  $\psi_\epsilon$ is an immersion which  induced metric $\Lambda_\epsilon$  and element of area $ dA_\epsilon$ are given by 
$$\Lambda_\epsilon= I -  \frac{4}{\epsilon} II + \frac{4}{\epsilon^2}III, \qquad dA_\epsilon = -\frac{1}{\epsilon^2}(\epsilon - 2 k_1)(\epsilon - 2 k_2) dA,$$ 
where $I$, $II$ and $III$ denote the first, second and third fundamental forms of $\psi$ and $dA$ is the element of area of $I$. Moreover the principal curvature functions of $\psi_\epsilon$, $ k_1^\epsilon$ and   $ k_2^\epsilon$ can be written as 
\begin{align*}
k_1^\epsilon &= \frac{\epsilon k_1}{\epsilon -  2k_1},\\
k_2^\epsilon &= \frac{\epsilon k_2}{\epsilon - 2 k_2}.
\end{align*}

Hence, the Gauss curvature, $K(\Lambda_\epsilon)$, of $\psi_\epsilon$ is given by 
\begin{equation} K(\Lambda_\epsilon) = \frac{\epsilon^2 K(I)}{(\epsilon - 2 k_1)(\epsilon - 2 k_2)} \geq 0,\label{ka}
\end{equation}
where $K(I)$ denotes the Gauss curvature of $I$ and writing  $\Lambda_\epsilon$ in an orthonormal reference of principal vector fields $\{e_1,e_2\}$, we have that 
\begin{equation} \Lambda_\epsilon \equiv \frac{1}{\epsilon^2}\left(
\begin{tabular}{cc}
$(\epsilon - 2 k_1)^2$ & 0\\
0 & $(\epsilon - 2 k_2)^2$ \end{tabular}\right).\label{pvf}\end{equation}

From \eqref{condition1}, \eqref{ka} and \eqref{pvf}, we deduce that 
$I\leq \Lambda_\epsilon$  
and   $\psi_\epsilon$  is   a complete immersion in $\R^3$ of non negative curvature. Now, by using the   Sacksteder theorem (see \cite{Sa}), either  $\psi_\epsilon(\Sigma)$  is a generalized cylinder or its Gauss curvature does not vanish   identically  and $\psi_\epsilon$ is a convex embedding satisfying one of the following items:
\begin{itemize}
\item[(A)]  $\Sigma$ is homeomorphic to a sphere, or
\item[(B)] $\Sigma$ is homeomorphic to a plane and, up to a motion in $\R^3$, there is a point $p_0\in \Sigma$ such that the plane $\{z=0\}$ is the tangent plane of $\psi_\epsilon(\Sigma)$   at $\psi_\epsilon(p_0)=(0,0,0)$ and the projection of $\psi_\epsilon(\Sigma)$ on $\{z=0\}$ is a convex domain ${\cal G}$ such that $\psi_\epsilon$ is a convex graph in the interior of ${\cal G}$, $\text{Int}({\cal G})$, and  a vertical segment at each point of ${\cal G}\setminus \text{Int}({\cal G})$. Moreover, if  $\{q_n\}$ is a divergent sequence of points in $\Sigma$, its height function $\{z(q_n)\}$ goes to infinity.
\end{itemize}

We will see that neither of these two cases can occur. In fact, the first case is not possible because any compact surface in $\R^3$ must have at least an elliptic point, which  contradicts our assumption about $\psi$.

In the second  case, it is clear that $\psi_\epsilon$ is a proper embedding. Actually, by the global convexity, there is  a cone $ {\cal V}_\epsilon  $ with axis of rotation  the $OZ$-axis such that $\psi_\epsilon (\Sigma) $ lies inside the cone ${\cal V}_\epsilon $. But $ \psi $ is obtained from $ \psi_\epsilon $ as a parallel surface to distance $ 2/\epsilon $,  thus $ \psi $  is also a proper map and  $ \psi (\Sigma) $ lies   inside the cone $ { \cal V } $ obtained as  the parallel surface  of $ { \cal  V }_\epsilon $ to a distance $2/\epsilon$.

Under these conditions we assert that $ \psi $ must have at least one elliptic point, which would also lead us to a contradiction. 
To see this, we can assume that, up to a vertical translation,  the vertex of $ { \cal V }$  is the origin and as $ \psi $ is proper and it is contained in $ { \cal V }$, the distance $ \psi (\Sigma) $ to the origin is a positive real number $ d_0 >  0 $.  Consider  the spherical cap $ \S^2_- (R, 2d_0)$ passing through the origin,  of height  $2d_0$ and  boundary the circle of radius $R$ obtained by the section  ${\cal V}\cap \{ z= 2d_0\}$. From the construction,  $\psi(\Sigma)\cap \S^2_- (R, 2d_0)=\emptyset$. Thus, fixing this  circle and  taking the spherical caps  $\S^2_-(R, 2 d_0 - t)$ of height $2d_0 - t$,  $0\leq t\le 2 d_0$ passing through $(0,0,t_0)$  with boundary the mentioned  circle, we get  a $t_0$, $d_0\leq t_0 < 2 d_0$  where $ \S ^ 2_ -(R,2d_0-t_0) $ intersects  the surface $\psi(\Sigma)$ for the  first time. It is clear that the  intersection points must be  interior and elliptic points of the surface.
 \end{proof}
 As a consequence of the previous theorem, we have the following result:
 \begin{cor}\label{cor1}
Let  $ \psi:  \Sigma \longrightarrow \R^3$ be a complete immersed surface with nonpositive Gauss curvature and  mean curvature $H$ verifying that $ |H|   \geq const  >0$. Then $ \psi (\Sigma) $ is a generalized cylinder in $ 
\R^3$.\end{cor}
\begin{proof}
 It follows  from the above theorem since the condition  imposed  on the mean curvature implies that one of the principal curvature functions $k_i$ satisfies $k_i^2 \geq \epsilon >0$ for some positive constant $\epsilon$.
 \end{proof}
\subsection{Complete special Weingarten surfaces}
Let   $ \psi : \Sigma \longrightarrow \R^3$ be an immersed surface with Gauss curvature $ K $ and mean curvature $H$.  $ \psi $ is  called a Weingarten surface if $H$ and $K$ are in a  functional relationship  $W(H,K)=0$. We say that  $\psi $ is a  {\sl special Weingarten surface} (in short, SW-surface) if there exists a ${\cal C}^1$-function  $ f: [0, \infty [ \longrightarrow \R $ such that 
\begin{align}
 H = f (H^ 2 -K),  \quad \hbox{  $f(0)\neq 0$},\label{we}
 \end{align}
 and 
\begin{equation}
4 t f '(t)^2 < 1, \qquad \forall t\in[0,\infty[. \label{ef}
\end{equation}
 A function satisfying \eqref{ef} is called an {\sl elliptic function}.
\vspace{.3\baselineskip }

 It is remarkable that for any elliptic function $f$ satisfying $f(0)\neq 0$, there is a  round sphere of radius $1/|f(0)|$  in the family of $SW$-surfaces associated to $f$.
  
  In  \cite{ST}  Sa Earp and Toubiana asked if the following result 
  \begin{quote}
  {\sc  Klotz and Osserman}\cite{KO}. A complete surface in $\R^3$ with constant mean curvature on which the Gauss curvature $K$ does not change sign is either a sphere, a minimal surface or a right circular cylinder,
  \end{quote}  can be extended to complete $SW$-surfaces. Here, we give an affirmative answer to this question and prove:
\begin{teo} Let  $ \psi:  \Sigma \longrightarrow \R^3$ be a complete $SW$-surface on which the Gauss curvature $K$ does not change sign. Then it is either a round sphere or a right circular cylinder.
\end{teo} 
\begin{proof}
Assume the mean curvature $H$ and the Gauss curvature $K$ of the immersion $\psi$ satisfy \eqref{we} for an elliptic function
$ f: [0, \infty [ \longrightarrow \R $. Then, if we denote  by $t=H^2-K$, the principal curvatures of $\psi$ can be written as
\begin{equation}
k_1(t)=  f(t) - \sqrt{t}, \qquad k_2(t)=  f(t) + \sqrt{t},\label{pc}
\end{equation}
and  from \eqref{ef}, the following expression is satisfied
\begin{equation}
-\frac{1}{2\sqrt{t} } < f'(t) <    \frac{ 1}{2\sqrt{t} }.    \label{ec}
\end{equation}
Moreover, up to a change of orientation if necessary, we can also assume that $f(0)>0$.
\vspace{.5\baselineskip }\\
{\sc Case I: $K\leq 0$.} In this case, as $f(0)>0$ we find $t_0>0$ such that $f(t_0) =\sqrt{t_0}$ and by integration in \eqref{ec}, we obtain
\begin{equation}
-\sqrt{t} + 2 \sqrt{t_0}< f(t)<-\sqrt{t},
\end{equation} 
that is, $k_2> 2 \sqrt{t_0}$ on $\Sigma$ (see Figure \ref{et}) and from Theorem \ref{th1}, $K$ vanishes identically.
But, from \eqref{ef}, $h(t) := t - f(t^2)$ is a strictly increasing function and  it has at most one zero. Thus,  any  complete $SW$-surface with vanishing Gauss curvature has constant mean curvature, that is, it must be a right circular cylinder, which concludes the proof in this case.
  \begin{figure}[H]
\begin{center}
\includegraphics[width=0.4\linewidth]{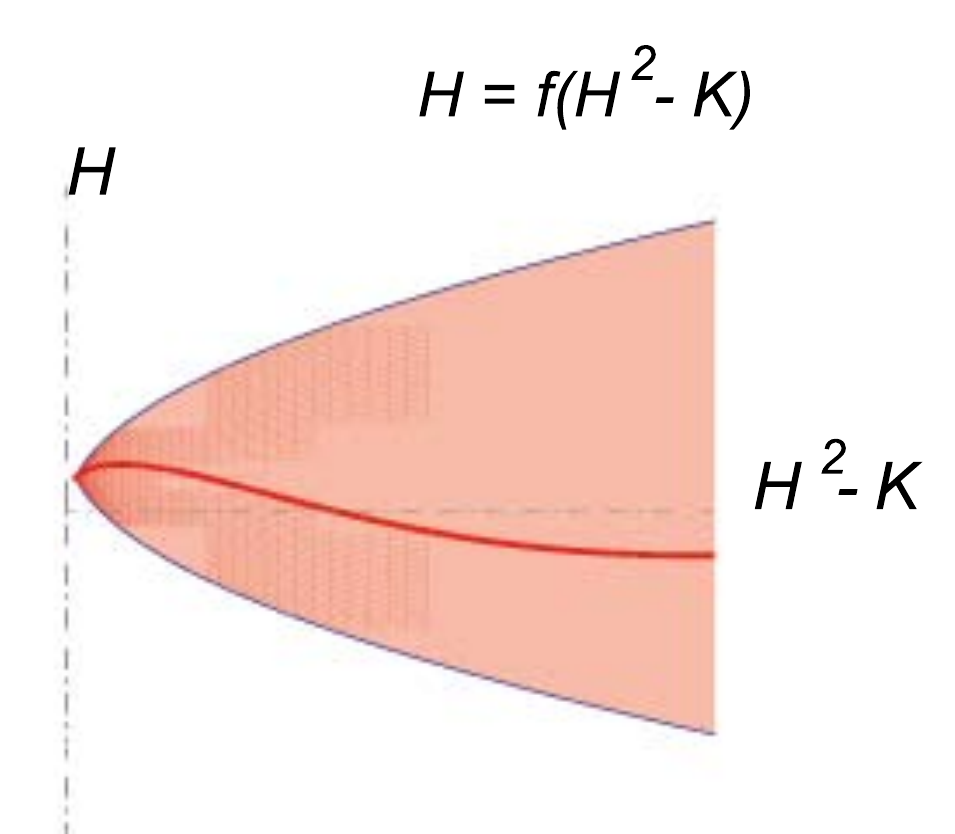}
\end{center}
\caption{SW-surface with $K\leq 0$.}
\label{et}
\end{figure}

\noindent {\sc Case II: $K\geq 0$.} This  case was proved in \cite[Theorem 5]{AEG}.
\end{proof}
\begin{remark}
{\rm  {\sc Case I} in the above result  
has also been proved in  \cite{AEG} on the additional assumption that $\psi$  is properly embedded}.
\end{remark}

\section{An extension of Efimov's theorem}
In this section we extend the results in \cite{E, E2} to complete surfaces with ends of negative Gauss curvature.  We will use the method developed by Efimov to prove  that in $\R^3$ it is impossible to have complete ${\cal C}^2$-immersed surfaces with Gauss curvature $K\leq c<0$. 

Throughout this section and as a standard notation, $S$ will denote a surface with a compact boundary $\partial S$ and $\Sigma$  will  denote a surface without boundary.
\begin{defi}
{\rm Assume that $\psi:S \longrightarrow \R^3$ is a ${\cal C}^2$-immersed surface with negative Gauss curvature $K<0$.  We say that the {\sl reciprocal value of the curvature of $\psi$ has variation with a linear estimate} if the following expression holds:
\begin{equation}
|\kappa(p) - \kappa(q)| \leq \epsilon_1 d_\psi(p,q) + \epsilon_2, \qquad \forall \  p,q\in S, \label{rvc}
\end{equation} 
for some non-negative constants $\epsilon_1$ and $\epsilon_2$, where
\begin{equation}\kappa = \frac{1}{\sqrt{-K}},\label{icur}
\end{equation}
and $d_\psi$ denotes the distance associated to the induced metric.}
\end{defi}
\begin{remark} \label{rk2}{\rm The above class includes the family of surfaces with Gauss  curvature  bounded above by a negative constant. In fact, if $K\leq -\epsilon <0$, then  \eqref{rvc} holds for $\epsilon_1=0$ and $\epsilon_2=2/\sqrt{\epsilon}$.}
\end{remark}
The purpose of this section is to prove the following results:
\begin{teo} \label{main}Let $S$ be a surface with a compact boundary $\partial S$ and $\psi:S \longrightarrow \R^3$ be a complete  ${\cal C}^2$-immersed surface with negative Gauss curvature. If the reciprocal value of the curvature of $\psi$ has variation with a linear estimate, then
$$\int_S|K| dA < \infty,$$
that is, $\psi$ has  finite total curvature and, in particular, $S$  is parabolic and has finite topology.
\end{teo}
\begin{teo} \label{ths41}Let $\psi:\Sigma \longrightarrow \R^3$ be a complete  ${\cal C}^2$-immersion with negative Gauss curvature outside a compact set $C\subset\Sigma$. If the reciprocal value of the curvature of $\psi$ has variation with a linear estimate outside $ C$, then
$$\int_\Sigma|K| dA < \infty,$$
that is, $\psi$ has  finite total curvature and, in particular, $\Sigma$ is parabolic and  has finite topology.
\end{teo}
\begin{teo} \label{ths42}Let  $\psi:\Sigma \longrightarrow \R^3$ be a complete ${\cal C}^2$-immersed surface with Gauss curvature $K\leq -\epsilon<0$ outside a compact subset $C$ of $\Sigma$. Then  $\Sigma$ has finite topology, $\psi$ is properly immersed and has finite area. Moreover, any end of $\psi$ is asymptotic to a half-line in $\R^3$ (see Figures \ref{f1a} and  \ref{f1b}).
\end{teo}
\begin{figure}[H]
\center{\includegraphics[width=0.4\linewidth]{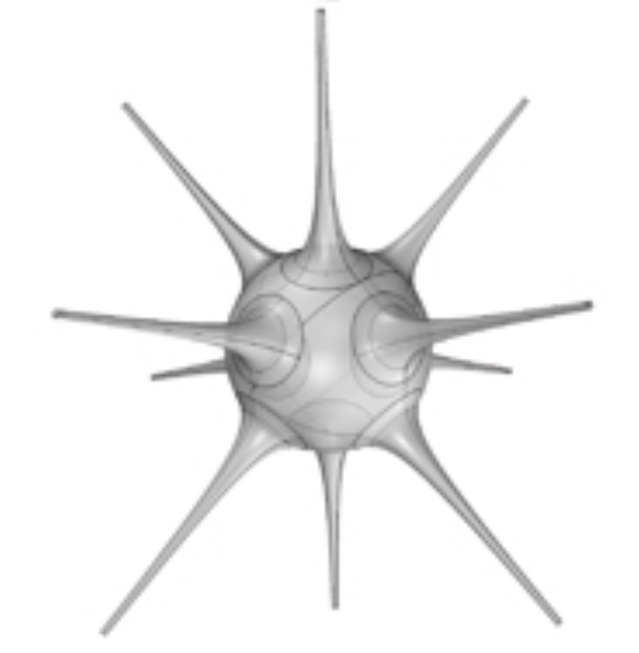}\hspace{1cm}
\includegraphics[width=0.4\linewidth]{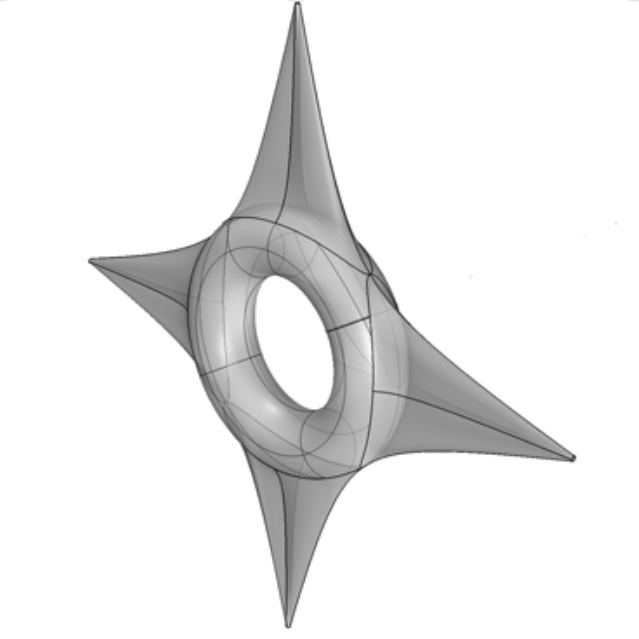}}
\caption{ Complete surfaces with negatively curved ends.}
\label{f1b}
\end{figure}
\subsection{The generalized lemmas}
The proof of Efimov's theorem and its extension (see \cite{E,E2}), are derived from a  result which concerns  special  ${\cal C}^1$-immersions $F:{\cal D} \longrightarrow \R^2$ where ${\cal D}\subset \R^2$ is constructed in \cite[Section 2.1]{KM} as any open simply-connected region containing $\Omega$, 
$$ \Omega =\{ (x,y)\in \R^2 | \ 0<x^2+y^2\leq \epsilon^2, \ \ y^2 \geq - c x\}, \qquad \epsilon,\  c >0,$$
and excluding the origin, (see Figure \ref{fdominio}).
\begin{figure}[H]
\center{\includegraphics[width=0.3\linewidth]{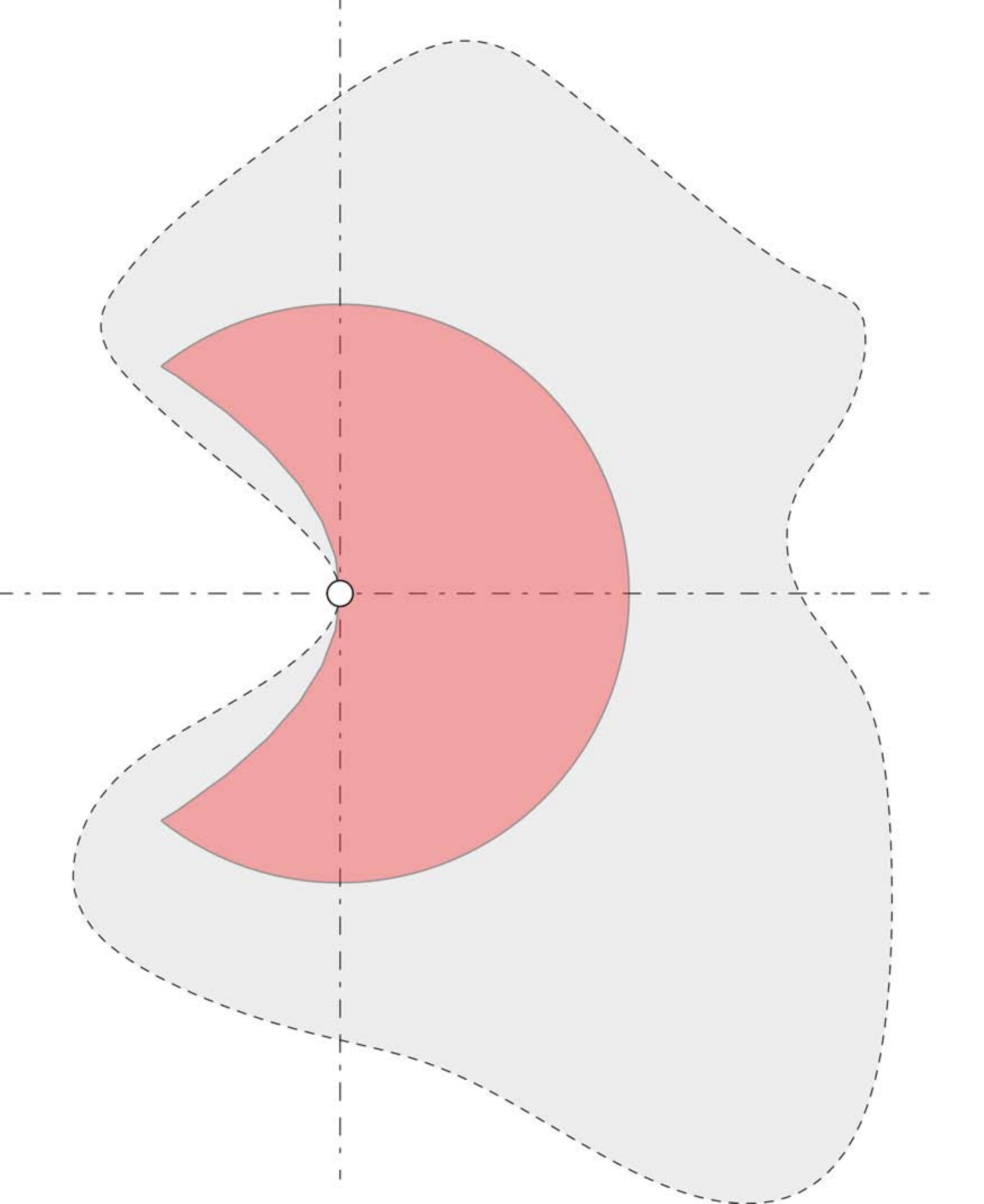}}
\caption{ Open simply-connected region containing $\Omega$.}
\label{fdominio}
\end{figure}
On  such a region ${\cal D}$ we can consider the   Riemannian metric $g^\star$ induced by $F$ (that is, the $g^\star$-length of an arc $\Gamma$ in ${\cal D}$ is just the Euclidean length of $F\circ \Gamma$) and its induced distance $d^\star$. On $\Omega$,  the following distance $d^\star_\Omega$ can also be considered:  $d_\Omega^\star(p,q)$ is defined as the infimum of the $g^\star$-lengths of rectifiable curves in $\Omega$ joining $p$ to $q$ for any $p,q\in \Omega$. It is clear that 
$$ d^\star(p,q) \leq d_\Omega^\star(p,q), \qquad \forall p,q\in \Omega.$$

Although in \cite{E2} Efimov proved a more general  version of the following  Generalized lemma \ref{gl1}, we will formulate the result in a way that it can  be used to prove the Generalized lemma \ref{gl2}. 
\begin{glemma} \label{gl1} Let  $F:{\cal D} \longrightarrow \R^2$   be a potential  given by  
$F=\nabla_0 f$, where   $f:{\cal D} \longrightarrow \R$ is a  ${\cal C}^2$-function  satisfying
\begin{align}
Det\left(\nabla_0^2f\right) &\leq -\frac{1}{g^2} < 0\label{cond1},\\
\left| g(p)-g(q) \right| & \leq \epsilon_1 d(p,q) + \epsilon_2, \qquad \forall p,q\in {\cal D},\label{cond2}
\end{align}
for some positive function  $g:{\cal D} \longrightarrow \R^+$ and $\nabla_0$  is the usual Euclidean gradient. Then the metric space $(\Omega,d^\star_\Omega)$ cannot be complete.
\end{glemma}
For a detailed  discussion and proof  of the  above Generalized lemma \ref{gl1} the reader is referred to  \cite[Generalized lemma A]{E2} and \cite[Main Lemma]{KM}.
\vspace{.3\baselineskip }

Now, before formulating   the Generalized lemma \ref{gl2}, we will  clarify some notation and terminology following the same approach as in \cite{KM}.

As we said at the beginning of this Section   by  $S$ we shall denote  a surface with a compact boundary $\partial S$ and $\psi:S\longrightarrow \R^3$ will be  a complete  ${\cal C}^2$-immersed surface  with negative Gauss curvature, $K<0$. It is not a restriction to assume that $S$ is orientable (in other case, we would work with its two-fold orientable covering).

Let $N:S \longrightarrow \S^2$ be the Gauss map of $\psi$ and  consider $III$  the third fundamental form of $\psi$, that is, the ${\cal C}^0$-metric induced by $N$. By $d_N$ we shall denote the distance associated to $III$.

Let  $(\widetilde{S}, III)$ be the completion of 
$(S, III)$ as a metric space and denote by  $\partial\widetilde{S}= \widetilde{S}\setminus S$ the  {\sl  boundary set of $\widetilde{S}$} and by $\widetilde{N}: \widetilde{S}  \longrightarrow \S^2$ the continuous extension of  $N$ to $\widetilde{S}$. As in \cite{KM}, we can also introduce the following concepts:
\begin{defi} {\rm Let $\Gamma$ be a non geodesic open circular arc on $\S^2$, $p\in \Gamma$ and $\epsilon >0$.  We consider  from each point of $\Gamma$ the open geodesic segment of $\S^2$ of length $\epsilon$ perpendicular to $\Gamma$ and directed on the side of concavity of $\Gamma$.  The region $R(\Gamma,\epsilon)$ formed by the union of $\Gamma$ and  all such segments is called an {\sl exterior rectangle} of $\Gamma$ at $p$.
 (See Figure \ref{f3}).
\begin{figure}[H]
\center{\includegraphics[width=0.4\linewidth]{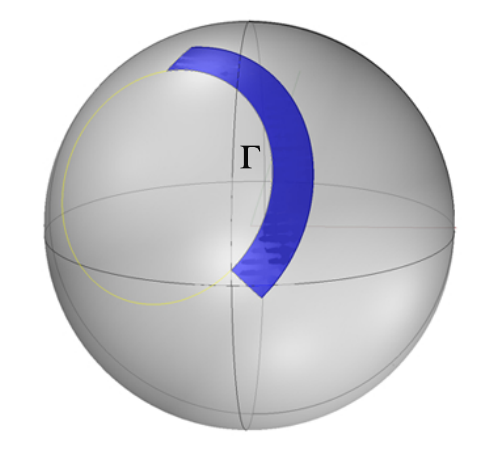}}
\caption{Exterior rectangle.}
\label{f3}
\end{figure}
We call  $(\widetilde{S}, III)$  {\sl concave at a point}  $p\in \partial\widetilde{S}$ if $p$ is in the closure  $\widetilde{U}$  (in $\widetilde{S}$) of an open region  $U \subset S\setminus\partial S$ such that $\widetilde{N}$ is one-to-one on $\widetilde{U}$ and $N(U)$ contains the interior of an exterior rectangle at $\widetilde{N}(p)$.
\vspace{.3\baselineskip }

If $(\widetilde{S}, III)$ is not concave at any point of $\partial\widetilde{S}$  we call  $(S,III)$  {\sl pseudo convex}. 
}\end{defi}
\begin{glemma}\label{gl2} Let $\psi:S \longrightarrow \R^3$ be  a complete  ${\cal C}^2$-immersion with negative Gauss curvature, $K<0$. If the reciprocal value of the curvature of $\psi$ has variation with a linear estimate, then $(S,III)$ is pseudo convex.
\end{glemma}
This result follows directly from the arguments of Efimov because the existence of a compact boundary, although it is  not considered by him, in no way alters his proof. Indeed, it follows  by taking into account  the discussion about  Lemma B (altered) in \cite[Subsection 22]{E2} and observing  that, even when the surface has a compact boundary,   it is possible to apply the Generalized lemma \ref{gl1} and complete the proof as in \cite[Subsection 35]{E}.

\subsection{Two auxiliary results}
In this subsection  we  discuss two results used  in the proof of  Theorem \ref{main}. Although the same notations and terminology as in the previous section is used,  we need to clarify some standard definitions and introduce further terminology.

Throughout,   we shall consider on  $S$ the Riemannian structure induced by the third fundamental form  $III$.  We shall choose a compact subset $C$ of $S$ such that $\partial S \subset \text{\rm Int}(C)$, by $\hat{\partial} C$ will denote the boundary set of $ S\setminus \text{\rm Int}(C)$ which is  given  by finitely many closed  ${\cal C}^1$-curves  and  $d_{N}$ will be  the distance associated to  $III$.
\begin{defi}
{\rm If, locally,  a parametrized arc on $S$ is a shortest path between any two of its points,  it is  called a {\sl geodesic  arc}. 
\vspace{.3\baselineskip }

As $N$ is a ${\cal C}^1$- isometric immersion, we can  check that any geodesic arc in $S$ of length less than $\pi$ is minimizing (i.e.,  it  is the  shortest path between any two of its points) and it is mapped by $N$ one-to-one onto a path of equal length along a great circle on $\S^2$.

If $p\in S\setminus \partial S$ and $\epsilon>0$, we shall denote by $D_\epsilon(p)$ the geodesic disc of radius $\epsilon$ in $S$, that is,
$$D_\epsilon(p) = \{ q\in S\setminus \partial S\ : \ d_N(p,q) < \epsilon\}.$$

$D_\epsilon(p)$ is called a {\sl full} geodesic disc if one can leave $p$ along a (half open) geodesic ray of length $\epsilon$ in every direction.

By a  {\sl convex} subset ${\cal H}$ in $S$ (or in $\S^2$) we understand  a non empty subset which satisfies  that any two of its points can  be joined by a unique minimizing geodesic  arc within  ${\cal H}$ (observe that with this definition $\S^2$ is not convex).
}\end{defi}

\begin{lem}  \label{l41}Consider  a complete  ${\cal C}^2$-immersion   $\psi:S \longrightarrow \R^3$ with negative Gauss curvature, $K<0$,  $r$   a positive real number such that $3r = d_{N}(\partial S, \hat{\partial} C)$ and  $p,q\in S\setminus  \text{\rm Int}(C)$  two points   satisfying 
\begin{equation} d_{N}(p,q) < \text{\rm min}\{ \text{\rm max$\{d_{N}(p,\hat{\partial} C), d_{N}(q,\hat{\partial} C)\}$}+ r, \pi\}.\label{dpq}
\end{equation}
If the reciprocal value of the curvature of $\psi$ has variation with a linear estimate, then there is a unique geodesic arc $\gamma$  from $p$ to $q$.
\end{lem}

 First, we remark the following assertion holds: 
 \begin{asser}
\label{asser1}{\rm Under the  hypotheses of  Lemma \ref{l41}, if  $\gamma$ is  a geodesic arc from $p$ to $q$ and  $D_\epsilon(p)$ and $D_\epsilon(q)$ are two full geodesic discs in $S\setminus\partial S$ satisfying  $l(\gamma)
 + 2 \epsilon<\pi$ with $2\epsilon < r$, then there is an open convex subset ${\cal H}$ in $S\setminus\partial S$ containing $D_\epsilon(p)\cup \gamma \cup D_\epsilon(q)$.}\end{asser}
 \begin{proof1}
Assume that $d_{N}(p,\hat{\partial} C)\geq d_{N}(q,\hat{\partial} C)$, then from \eqref{dpq},  $D_\epsilon(p)\cup \gamma \cup D_\epsilon(q)$ lies on the geodesic disc $\D=D_{d_N(p,q)+2\epsilon}(p)$ which is contained   in  $S \setminus\partial S$. Moreover,  the convex subset  ${\cal H}$ can be  constructed within  $\D$ using  the same arguments as in \cite[{\sc Observation 4}, item (B)]{KM} and applying our generalized lemma \ref{gl2} (instead of Lemma A of \cite{KM}). We detail here, for the reader's benefit, how this construction  can be done.
\vspace{.5\baselineskip }\\
{\sc Case I.} First, we make the construction of ${\cal H}$  under the additional assumption that the closures of $D_\epsilon(p)$ and $D_\epsilon(q)$ in $\widetilde{S}$ lie within $S\setminus \partial S$.

Let ${\cal T}(\tau)$ be the closed tubular neighborhood of $\gamma$ of radius $\tau$ inside $\D$. As $N$ is a local diffeomorphism, it is clear  there exists $\tau>0$ such that $N$ is one-to-one on 
$$\overline{D}_\epsilon(p)\cup {\cal T}(\tau)\cup \overline{D}_\epsilon(q),$$ where by bar we denote the corresponding closure in $S$. Consider 
$$ \hat{\tau}= \sup \{ \tau \in ]0,\epsilon[ \ :   \text{ N is one-to-one on} \ \overline{D}_\epsilon(p)\cup {\cal T}(\tau)\cup \overline{D}_\epsilon(q)\},$$
then it is easy to see  that $\hat{\tau} = \epsilon$, otherwise we find somewhere on the metric closure of ${\cal T}(\hat{\tau})$ a point $\tilde{p}\in \partial \widetilde{S}$ such that  $N(\tilde{p})$ lies on a non geodesics circle $\Gamma$ on the boundary of $N({\cal T}(\hat{\tau}))$ parallel to $N\circ\gamma$  with its center on the opposite side of $\Gamma$ from $N({\cal T}(\hat{\tau}))$. But then  $\widetilde{S}$ is  concave at $\tilde{p}$ which gives a contradiction with the Generalized Lemma \ref{gl2}.

Now, we will consider that $N\circ \gamma$ parametrizes some portion of the equator $\{y=0\}$ on $\S^2$ with its midpoint at $(0,-1,0)$ and such that a certain $y_0<0$ is the $y$ coordinate at the points $N(p)$ and $N(q)$.   We will denote by ${\cal R}$  the right elliptical cylinder in $\R^3$ formed by the union of all lines parallel to the $x$-axis through the boundaries of $D_\epsilon(N(p))$ and  $D_\epsilon(N(q))$. Let $\Pi_\tau^+$ and $\Pi_\tau^-$ be the planes in $\R^3$ making an angle $\tau$, $\tau\in [0,\pi/2]$,  with $z=0$  and tangent to ${\cal R}$ along a line on which $y=const \geq y_0$ with $z= const \geq 0$ and $z=const\leq 0$, respectively,  see Figure  \ref{convex}(a). 

It is clear there exists a unique $\tau_0\in ]0,\pi/2[$ for which $\Pi_{\tau_0}^+$ and $\Pi_{\tau_0}^-$ pass through the origin in $\R^3$ and cut $\S^2$ along great circles tangent to the circular boundaries of $D_\epsilon(N(p))$ and  $D_\epsilon(N(q))$.

On $\S^2$ we take  the neighborhood ${\cal E}_\tau$ of $N\circ \gamma$ formed by the open region in $y<0$ lying below $\Pi_\tau^+$, above $\Pi_\tau^-$, to the  right of  $D_\epsilon(N(p))$  and to the left  of  $D_\epsilon(N(q)))$, see the blue region in Figure  \ref{convex}(b).

From the first part of the construction, $N$ is one-to-one on $D_\epsilon(p)\cup {\cal T}(\epsilon)\cup D_\epsilon(q)$ and $N$ maps $D_\epsilon(p)\cup {\cal T}(\epsilon)\cup D_\epsilon(q)$ onto ${\cal E}_0$.  Let $\hat{\tau}$ be  the supremum  of all $\tau$ values  in $[0,\tau_0]$ for which some neighborhood ${\cal N}_\tau$ of $\gamma$ within $\D$ is mapped  by $N$ one-to-one onto ${\cal E}_\tau$,  then we can prove that $\hat{\tau}=\tau_0$.
 Otherwise we find a point $\tilde{p}\in \partial \widetilde{S}$ on the metric closure  of ${\cal N}_{\hat{\tau}}$. But, under our assumption that the closures of $D_\epsilon(p)$ and $D_\epsilon(q)$ in $\widetilde{S}$ lie within $S\setminus \partial S$, $N(\tilde{p})$ is to distance greater than $\epsilon$ from $N(p)$ and $N(q)$. Thus, $N(\tilde{p})$ lies along the $\Gamma = \Pi_{\hat{\tau}}^+\cap \S^2$ or  $\Gamma = \Pi_{\hat{\tau}}^-\cap \S^2$  which centers lie on the opposite  side of $\Gamma$ from ${\cal E}_{\hat{\tau}}$ and $\widetilde{S}$ is concave at $\tilde{p}$ which  contradicts the Generalized lemma \ref{gl2}.

By taking ${\cal H}={\cal N}_{\tau_0}$, it is clear that  ${\cal H}$ is a convex subset within $\D$ and $N: {\cal H}\rightarrow {\cal E}_{\tau_0}$ is one-to-one.
\vspace{.5\baselineskip }\\
{\sc Case II.} In the  general case,  we can apply  for any $\tau\in ]0,\epsilon[$  the Case I to $D_\tau(p)\cup \gamma \cup D_{\tau}(q)$ to get an open convex subset ${\cal H}_\tau$ within $\D$ and containing $D_\tau(p)\cup \gamma \cup D_{\tau}(q)$. Then ${\cal H}$ can be constructed by taking 
$$ {\cal H} = \bigcup_{0<\tau<\epsilon}{\cal H}_\tau.$$
\end{proof1}
 \begin{figure}[H]
\center{ \includegraphics[width=0.39\linewidth]{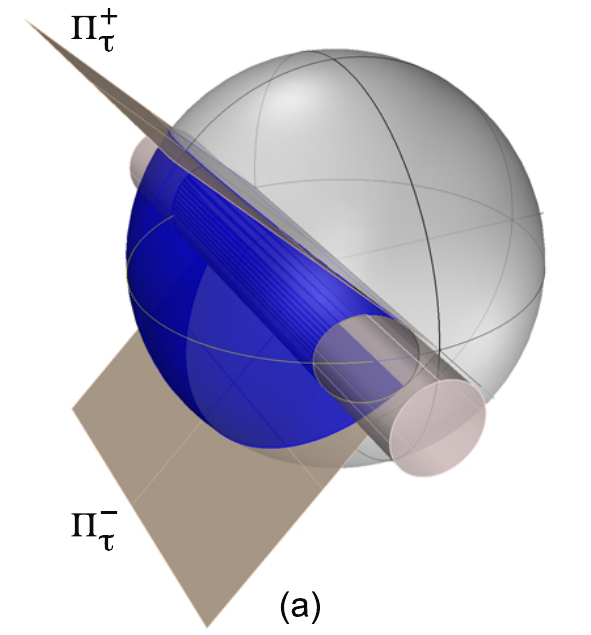}\hspace{1cm}
\includegraphics[width=0.4\linewidth]{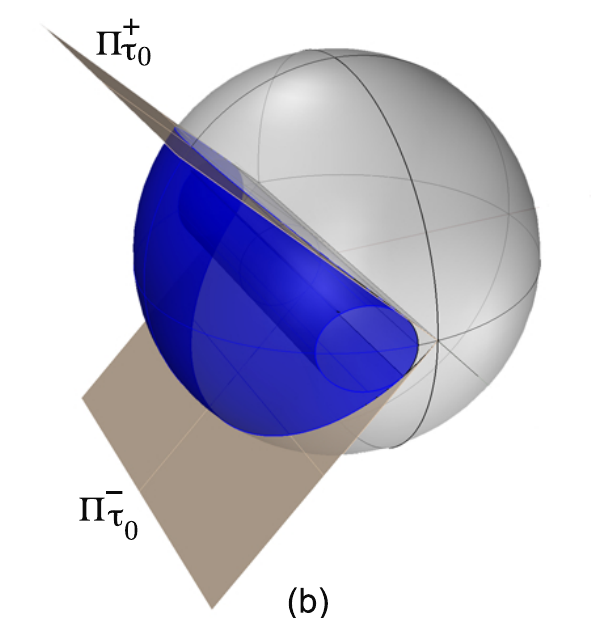}}
\caption{Construction of the convex subset $N({\cal H})$.}
\label{convex}
\end{figure}
\begin{proof2}
To prove the Lemma we suppose that $p\neq q$, otherwise it is trivial. Then,  from \eqref{dpq} and  since $S\setminus\partial S$ is connected,  we can take a parametrized arc $\Gamma$ satisfying 
$$ l_N(\Gamma) <\text{\rm min}\{ \text{\rm max$\{d_{N}(p,\hat{\partial} C), d_{N}(q,\hat{\partial} C)\}$}+ r, \pi\},$$
 where $l_N$ denotes  the length induced by the third fundamental form $III$. 
 
 But $N$ is a local isometry and $\Gamma$ is compact, thus  there is  a positive  real number $\epsilon$,  $0< \epsilon <r$ such that,
 \begin{equation}
  l_N(\Gamma) + 2\epsilon <\text{\rm min}\{ \text{\rm max$\{d_{N}(p,\hat{\partial} C), d_{N}(q,\hat{\partial} C)\}$}+ r, \pi\},\label{lgg}
\end{equation}
 and $D_\epsilon(x)$ is a full geodesic disc in $S \setminus\partial S$ for any point $x$ of $\Gamma$. 
 
 Using again the compacity of $\Gamma$ we can fix $x_0=p, \ x_1,\cdots , \ x_n=q$ points in $ \Gamma$, ordered by the parametrization of $\Gamma$ and  satisfying 
  \begin{equation}
 d_N(x_k,x_{k+1}) < \epsilon, \qquad k=0,\cdots, n-1.\label{xxx}
\end{equation}
 
 As $x_1\in D_\epsilon(x_0)$ it is clear,  there is a  unique geodesic arc $\gamma_0$ within  $S \setminus\partial S$  joining $x_0$ to $x_1$ with
 $$ l_N(\gamma_0)  + 2 \epsilon < \pi,$$
 and, following the same ideas as in \cite[{\sl Proof of Lemma B}]{KM},  we can apply an  induction argument on the fixed number of points.
 More specifically, assume there is a unique  geodesic arc $\gamma_{k-1}$ within  $S \setminus\partial S$  from  $x_0$ to $x_{k-1}$ with
 $$ l_N(\gamma_{k-1})  + 2 \epsilon < \pi.$$
 Then we prove the existence of a minimizing geodesic arc $\gamma_k$  from $p$ to $x_k$ satisfying 
\begin{equation} l_N(\gamma_{k})  + 2 \epsilon < \pi.\label{lgg2}
\end{equation}
 In fact, by  applying  the Assertion \ref{asser1} to $\gamma_{k-1}$, there is an open  convex set ${\cal H}_{k-1}$ in  $S \setminus\partial S$  containing 
 $D_\epsilon(p)\cup \gamma_{k-1} \cup D_\epsilon(x_{k-1})$, but from \eqref{xxx}, $x_{k}\in D_\epsilon(x_{k-1})$ and  we find  a minimizing geodesic arc $\gamma_k$  from $p$ to $x_k$ within ${\cal H}_{k-1}$. Moreover,  it is clear from \eqref{lgg} that \eqref{lgg2} holds,
  which concludes the proof. 
\end{proof2}
\begin{lem} $d_N(p,\hat{\partial} C) <\pi$ for any  $p\in S \setminus C$.\label{l42}
\end{lem} 
\begin{proof}
We argue by contradiction. If  the lemma does not hold, then from the compacity of  $C$ we can suppose there are  $p\in S \setminus C$ and $q\in \hat{\partial} C$ such that  $d_N(p,q)=d_N(p, \hat{\partial} C)=\pi$.

Take  a positive real number  $\epsilon < \min\{\pi/2,r\}$, such that $D_{2\epsilon}(q)$ is a full geodesic disc  in $S\setminus\partial S$ and  consider the circle $\S_\epsilon(q)=\partial D_\epsilon(q)$.  Then,  by fixing $q_1\in \S_\epsilon(q)$ satisfying 
$$ d_N(p,q_1) = d_N(p,\S_\epsilon(q)) = \pi-\epsilon < \pi = \text{\rm min}\{ \text{\rm max$\{d_{N}(p_1,\hat{\partial} C), d_{N}(q,\hat{\partial} C)\}$}+ r, \pi\},$$
we can  apply the Lemma \ref{l41} to $q_1$ and $p$ and  find  a minimizing geodesic $\gamma_1$ in 
$S\setminus\partial S$ from $q_1$ to $p$.  The geodesic ray from $q$ to $q_1$ together $\gamma_1$ is a  minimizing geodesic arc  $\gamma$ in $S\setminus\partial S$  joining $q$ to $p$ with   $l_N(\gamma) = \pi$.

Let $m$ be the midpoint of $\gamma$. For any $t$, $0<t<\pi/2$, we denote by $p_t$ and $ q_t$ the points in $\gamma$  satisfying
$$ d_N(p,p_t) = d_N(q,q_t) = t.$$

Throughout let us assume that $N(m)=(0,-1,0)$, $N(p) = (-1,0,0)$, $N(q)=(1,0,0)$ and  $\gamma$ is mapped one-to-one into the corresponding geodesic arc of $\S^2$ in the plane $\{ z=0\}$.

We also choose $\epsilon_t >0$,  $\epsilon_t< \text{\rm min}\{t,r, \pi/2\}$ such that  $D_{\epsilon_t}(p_t)$, $D_{\epsilon_t}(q_t)$ are full geodesic discs.  Under these conditions we can apply the Assertion \ref{asser1} and prove that there is an open convex subset ${\cal H}_t$ in $S\setminus\partial S$ verifying that $N$ maps ${\cal H}_t$ one-to-one onto the convex subset $N({\cal H}_t)$. But for the construction of ${\cal H}_t$ (see the proof of Assertion \ref{asser1} for more details) we can check  that, see also Figure \ref{convex},
\begin{equation} \lim_{t\rightarrow 0} N({\cal H}_t) = D_{\pi/2}(N(m))= \S^2\cap \{y<0\},
\label{lht}\end{equation} and  $N$ maps  $D_{\pi/2}(m)$ one-to-one  onto $D_{\pi/2}(N(m))$ in $\S^2$, see Figure \ref{ff}.

At this point, we can take a positive real number   $\epsilon_1<\min\{\pi/2,r\}$,  such that $D_{2\epsilon_1}(p)$ is a full geodesic disc and two different points $p_1, p_2\in D_{2\epsilon_1}(p)$  to a distance $\epsilon_1$ from $p$, which are mapped by $N$ into the points $N(p_1)$ and $N(p_2)$ lying on the geodesic arc in the plane $\{ y=0\}$ in such a way that  $N(p_1)$ lies in the half space $\{z>0\}$ and $N(p_2)$  is lying in the half space $\{z<0\}$, see Figure \ref{ff}. 
 \begin{figure}[H]
\center{ \includegraphics[width=0.5\linewidth]{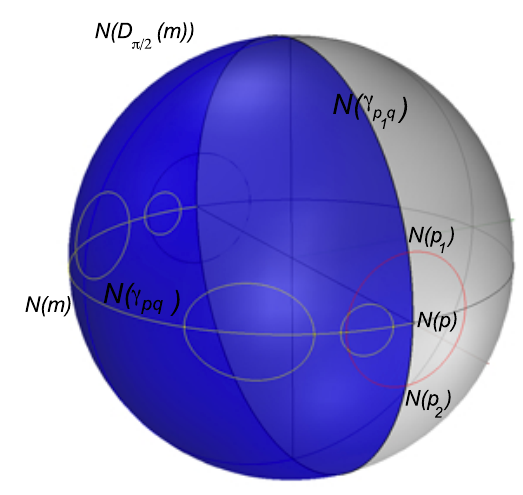}}
\caption{Proving the Lemma}
\label{ff}
\end{figure}
By the choice of $p_1$ and $p_2$ we have $\pi-\epsilon_1\leq d_N(p_i,q)$, $i=1,2$. Moreover, since $N$ is a global isometry from  $D_{\pi/2}(m)$ onto $\{y<0\}$, there exist two sequences of curves $\{\Gamma^i_n\}$ joining $q$ to $p_i$, $i=1,2$ which interior points lie within  $D_{\pi/2}(m)$ and such that $l_N(\Gamma_n^i)\rightarrow \pi-\epsilon_1$.
 Thus,
$$d_N(p_i,q) = \pi-\epsilon_1<\pi= \min\{\max\{d_N(p_i,\hat{\partial}C),d_N(q,\hat{\partial}C)\} + r, \pi\}, \quad i=1,2,$$
and   we can apply  Lemma \ref{l41} to prove  the existence of minimizing geodesic arcs in  $S\setminus\partial S$, $\gamma_{p_1q}$ and $\gamma_{p_2q}$ from $p_1$ to $q$ and from $p_2$ to $q$, respectively.   But,  having in mind that there is a unique geodesic in $\S^2$ of length $\pi-\epsilon_1$ from $p_i$ to $q$, $i=1,2$, we have that $N(\gamma_{p_1q})$ lies in $\{y=0\}$ on the northern hemisphere of $\S^2$ and  $N(\gamma_{p_2q})$ lies in $\{y=0\}$ on the southern hemisphere of $\S^2$. 

Now, consider the following closed subset $A_i$  in $\gamma_{p_iq}$, 
$$A_i=\{x\in\gamma_{p_1q}\ | \ d_N(x,m) = \pi/2\}, \qquad i=1,2$$
It is clear that  there is a neighborhood of $p_i$ in $A_i$. Moreover, by using that  $N$ is a local diffeomorphism, it follows that $A_i$ is also an open subset of $\gamma_{p_iq}$ and so $A_i=\gamma_{p_iq}$, $i=,1,2$.
In other words, there is no point of $\partial \widetilde{S}$ to a distance $\pi/2$ from $m$,  the closure  $\overline{D}_{\pi/2}(m)$ of $D_{\pi/2}(m)$  lies in $S\setminus\partial S$ and $N$ is one-to-one in $\overline{D}_{\pi/2}(m)$. 
Since  $N$ maps $\overline{D}_{\pi/2}(m)$ one-to-one onto the  eastern hemisphere of $\S^2$ while $\overline{D}_{\pi/2}(m)$  is compact there is $\epsilon_2>0$ such that $D_{\pi/2 + \epsilon_2}(m)$ is a full geodesic disc and $N$ maps it one-to-one  onto $N(D_{\pi/2 + \epsilon_2}(m))$. But $q\in \hat{\partial} C$ and $\hat{\partial} C$ is a finite set of regular curves in $\Sigma$, then  by the above construction, we can assert that   there are points of $N(\hat{\partial} C)$ in  $N(D_{\pi/2 + \epsilon_2}(m))$ which  distance from $p$ is less than $\pi$. This fact  is a  contradiction with  the assumption that  $\pi = d_N(p,\hat{\partial} C)$. \end{proof}
\subsection{ Proof of the Theorems \ref{main}, \ref{ths41} and \ref{ths42}.}
\begin{proof3}
To prove Theorem \ref{main} we observe that from Lemma \ref{l42}, $d_N(p,\hat{\partial}C)< \pi$, for any $p\in S\setminus C$,   and from Lemma \ref{l41},  there exists a  minimizing geodesic arc, $\gamma_p$,  joining $p$ to $\hat{\partial} C$. By the minimizing property, $\gamma_p$ meets orthogonally to $\hat{\partial} C$. Thus, the area of $\psi(S)$ respect to the third fundamental form $III$ is given by 
$$ A_N(\psi(S)) = A_N(\psi(C)) + \int_0^{L_N(\hat{\partial} C)} l_N(\gamma_q)  < A_N(\psi(C))  + L_N(\hat{\partial} C) \pi < \infty, $$
where $L_N(\hat{\partial} C)$ is the length of $\hat{\partial} C$ respect to $III$,
which concludes the proof.\end{proof3}
\begin{proof4}
It follows directly from  Theorem \ref{main}.
\end{proof4}
\begin{proof5}
From Theorem \ref{main}, Remark \ref{rk2} and the stated hypothesis in the theorem, the area of $\psi$, $A(\psi(\Sigma)) $,  is estimated by
 $$A(\psi(\Sigma)) \leq \frac{1}{\epsilon}\int_\Sigma |K| dA < \infty.$$

That is, the immersion has finite area. Thus, we are left to show that every end of $\Sigma$  is properly immersed and asymptotic  to a half-line. 

Let us consider an end of $\Sigma$, which we will assume parametrized on the set $E=\{p\in\R^2|\ 0<|p|\leq1\}$. We can also assume that the curvature is non positive for every point on $E$. Since the area of the end is finite, there exists a strictly decreasing sequence of radii $\{\epsilon_n\}$ going to zero such that the curves $\Gamma_n=\psi(\{p\in\R^2|\ |p|=\epsilon_n\})$ satisfy that their length  
\begin{equation}\label{final1}
\{l(\Gamma_n)\}\rightarrow 0.
\end{equation} 

For $n<m$ denote by ${\cal A}_n^m=\psi(\{p\in\R^2|\ \epsilon_m\leq |p|\leq\epsilon_n\})$ and ${\cal A}_n^{\infty}=\psi(\{p\in\R^2|\ 0< |p|\leq\epsilon_n\})$. 
Since the end has non positive curvature at every point, then
\begin{equation}\label{final2}
{\cal A}_n^m\subseteq conv(\Gamma_n\cup \Gamma_m)=conv(conv(\Gamma_n)\cup conv(\Gamma_m))
\end{equation}
(see, for instance, \cite{O2}), where $conv(\cdot)$ denotes the convex hull of a set in $\R^3$. 

Thus, as the end can not be bounded \cite{Bu},  $\cup_{n=1}^{\infty}\Gamma_n$ is unbounded. From this fact and (\ref{final1}), passing to a subsequence if necessary, we can assume  
\begin{equation}\label{final3}
\max\{|p|\,|\ p\in \Gamma_{n}\}<\min\{|p|\,|\ p\in \Gamma_{n+1}\}\quad\text{with}
\quad\left\{\min\{|p|\,|\ p\in \Gamma_{n}\}\right\}\rightarrow\infty.
\end{equation}

Now, for all $n$ we consider two points $q_n\in conv(\Gamma_1)$ and $p_n\in conv(\Gamma_{n+1})$. Then, passing to a subsequence if necessary, we can suppose   there exists a unit vector $v_0$ in $\R^3$ such that 
$$
\left\{\frac{p_n-q_n}{|p_n-q_n|}\right\}\rightarrow v_0.
$$

Since $\{|p_n-q_n|\}\rightarrow\infty$, it is easy to check that the vector $v_0$ does not depend neither on
the chosen points $q_n$ because $conv(\Gamma_1)$ is a bounded set, 
nor on the chosen points $p_n$ because the diameter of $conv(\Gamma_{n+1})$ goes to zero, from (\ref{final1}).

Let us define the solid cylinders
$$
{\cal R}_1^+=\{q+t v_0|\ q\in conv(\Gamma_1),t\geq0\},\qquad {\cal R}_1=\{q+t v_0|\ q\in conv(\Gamma_1),t\in\R\}.
$$
And let us prove that ${\cal A}_1^{\infty}\subseteq{\cal R}_1^+$. From this condition, (\ref{final2}) and (\ref{final3}), we will have that the end is properly immersed.

Assume ${\cal A}_1^{\infty}\not\subseteq{\cal R}_1^+$ then, from (\ref{final2}), there exists $n_0>1$ and a point $x_0\in conv(\Gamma_{n_0})$ such that $x_0\not\in {\cal R}_1^+$. Hence the compact set
$$
\hat{C}=\left\{\frac{x_0-q}{|x_0-q|}\in\S^2\,|\ q\in conv(\Gamma_1)\right\}
$$
does not contain to the vector $v_0$.

Using (\ref{final2}), for each $n>n_0$ there exist two points $q_n\in conv(\Gamma_1)$ and $p_n\in conv(\Gamma_{n+1})$ such that $x_0=(1-t_n)q_n+t_np_n$, for some $t_n\in[0,1]$. In such a case $(p_n-q_n)/|p_n-q_n|\in \hat{C}$ which contradicts that the limit of this sequence must be $v_0$.

Once we have proven that ${\cal A}_1^{\infty}\subseteq{\cal R}_1^+$ if, analogously, we define
$$
{\cal R}_n^+=\{q+t v_0|\ q\in conv(\Gamma_n),t\geq0\},\qquad {\cal R}_n=\{q+t v_0|\ q\in conv(\Gamma_n),t\in\R\},
$$
it is elementary to check that 
\begin{equation}\label{final7}
{\cal A}_n^{\infty}\subseteq{\cal R}_n^+, \qquad \text{ for all $n>1$}.
\end{equation}

Since ${\cal R}_{n+1} \subseteq {\cal R}_n$, then  from (\ref{final1}) we have that   ${\cal R}=\cap_{n=1}^{\infty}{\cal R}_n$ is a line, and (\ref{final7}) proves that the end is asymptotic to ${\cal R}$ as we want to show. 
\end{proof5}
\section{Complete surfaces with non positive extrinsic curvature in $\H^3$ and $\S^3$}
Fundamental results of surfaces' theory  in $\R^3$ essentially, only depend on the Codazzi equation which yields true if we consider  any other space form. This is the case, for example, of Hopf's theorem on the classification of constant mean curvature spheres or Liebmann's theorem about surfaces of positive  constant Gauss curvature. Thus, it is not surprising that various results from theory of immersed surfaces can be proved in the abstract setting of Codazzi pairs, that is,   pairs $(I,II)$ of real quadratic forms on an abstract surface, where $I$ is a Riemannian metric and $II$ satisfies the Codazzi-Mainardi equations of the classical surface theory with respect to the metric $I$.

In this section,  we use the  theory of Codazzi pairs to give  Efimov and Milnor's type results on surfaces in non euclidean space forms. 
The main  idea in this sense is to study  Codazzi pairs $(I,II)$  as a geometric object in a non standard way on a Riemannian surface $(\Sigma, I)$ of non positive Gauss curvature and use this study to deduce consequences  when  $(I,II)$ are  the first and second fundamental forms of a surface immersed in  a space form.
We shall follow the same approach as introduced by Aledo, Espinar and G\'alvez,  \cite{AEG}. 

As in the previous sections, we shall  assume that  $\Sigma$ is an oriented surface (otherwise we would work with its oriented two-sheeted covering). Moreover, throughout we always consider a ${\cal C}^\infty$-differentiability.
\begin{defi}{\rm 
A fundamental pair on $\Sigma$ is a pair of real quadratic forms $(I,II)$ on $\Sigma$, where $I$ is a Riemannian metric. The shape operator $A$ of $(I,II)$ is defined by 
\begin{equation}
II(X,Y) = I(AX,Y), \qquad X,Y\in T\Sigma.\label{so}
\end{equation}

We also define the {\sl mean curvature}, the {\sl extrinsic curvature} and the {\sl principal curvatures} of the pair $(I,II)$ as one half of the trace, the determinant and the eigenvalues of the endomorphism $A$.
It is remarkable that, in general, there is not any connection between the extrinsic curvature of a fundamental pair $(I,II)$ and the Gauss curvature $K(I)$ of the Riemannian metric $I$.
\vspace{.5\baselineskip }

We say that the principal curvatures $k_1$ and $k_2$ of a fundamental pair on $\Sigma$ are {\sl strictly separated} if there exist real numbers $c_1$ and $c_2$ such that
\begin{equation}
k_1 \leq c_1<c_2\leq k_2,\label{cpss}
\end{equation}
 on $\Sigma$.
 }
 \end{defi}
\begin{defi} {\rm Let $(I,II)$  be a fundamental pair, we  say  that $(I,II)$ is a {\sl Codazzi pair} if the following equation holds
\begin{equation}
\nabla_XAY - \nabla_YAX - A[X,Y]=0,\label{ce}
\end{equation}
for any vector field $X,Y\in T\Sigma$, where $\nabla$ is the Levi-Civita connection of $I$.}
\end{defi}
Codazzi pairs appear in a natural way in the study of surfaces. For instance, the first and second fundamental forms of any  surface isometrically immersed in a 3-dimensional space form is a Codazzi pair and the same happens for spacelike surfaces in a 3-dimensional Lorentzian space form. 

In general,  if  a immersed surface in an n-dimensional (semi-Riemannian) space form has a  parallel unit normal vector field $\xi$, then   the first fundamental form and the second fundamental form associated with $\xi$ constitute a Codazzi pair.   Many other examples of Codazzi pairs also appear in \cite{AEG2, AEG,BB,KM3,KM4,OS} and references therein. 

Although we will apply our results on Codazzi pairs  to surfaces in a 3-dimensional space form, all  the above mentioned comments  show that the results can also be applied to many others different contexts.

\subsection{Codazzi pairs on complete surfaces with non positive curvature.}
In this subsection we shall  prove the following result:
\begin{teo}\label{at1}
Let $(I,II)$ be a Codazzi pair on $\Sigma$ with strictly separated principal curvatures. If $(\Sigma,I)$ is a complete surface with Gauss curvature $K(I)\leq 0$, then only one of the following items  hold:
\begin{itemize}
\item $I$ is a flat metric and $\Sigma$ is homeomorphic either a plane,  or a cylinder or  a flat torus.
\item $I$ is not flat, $\Sigma$ is homeomorphic to a plane and 
\begin{equation}
\int_\Sigma |K(I)| \ dA_I \leq 2 \pi.\label{cf}
\end{equation}
\end{itemize}
\end{teo}
\begin{proof}
Consider the third fundamental form $III$  associated with  the pair $(I,II)$, which is given by 
\begin{equation}
III(X,Y) = I(AX,AY),\qquad  X,Y\in T\Sigma,\label{tff}
\end{equation}
and take  local doubly orthogonal coordinates   $(u,v)$
so that 
\begin{align}
I &= E du^2 + G dv^2,  \nonumber
\\II &= k_1E du^2 + k_2 G dv^2, \label{doc}\\
III &= k_1^2E du^2 + k_2^2G dv^2.\nonumber
\end{align}
Such doubly orthogonal coordinates are locally available on an open dense subset of $\Sigma$ and we can use them to check  identities are  valid on  all the surface.

If  $a\in \R\setminus \{0\}$ satisfies that $a k_1\neq 1$, $ak_2\neq 1$ on $\Sigma$ , then from \eqref{doc},
the quadratic form 
$$ \Lambda_a = I- 2 a II+ a^2 III$$   is a Riemannian metric   given, locally,  by
\begin{equation}  \Lambda_a  = (1 - a k_1)^2 du^2 + (1- a k_2)^2 dv^2, \label{num}
\end{equation}
which Gauss curvature, $K( \Lambda_a)$,  can be written, see \cite{KM2},  as:
\begin{equation}  K(\Lambda_a) = \frac{K(I)}{(1-a k_1) (1-a k_2)}.\label{cnm} 
\end{equation}
From \eqref{cpss}, we can choose $a\in \R\setminus \{0\}$ so that 
$$ k_1\leq c_1 < \frac{1}{a} < c_2 \leq k_2, $$
and, if we take $ c_0 = \min\{|1 - a c_i| : \ i=1,2\}$, then from \eqref{num} and \eqref{cnm}, the  following expressions hold,
\begin{align*}
(1-ak_i)^2&\geq c_0^2, \quad \text{ i=1,2}\\
\Lambda_a &\geq c_0^2 I, \\
K(\Lambda_a) &\geq 0.
\end{align*}
That is, $(\Sigma,\Lambda_a)$ is a complete Riemannian surface of non negative curvature.

We  distinguish two cases:
\vspace{.3\baselineskip }\\
{\sc Case I}: $K(\Lambda_a)$ vanishes identically. In this case,  from \eqref{cnm}, $I$ is also a flat metric. Thus,  if $\overline{\Sigma}$ denotes the universal cover of $\Sigma$,    we have, from  Cartan's theorem,  that $(\overline{\Sigma},I)$  is isometric to the usual euclidean plane $\R^2$.  But then,  we deduce that $\Sigma$ is homeomorphic to $\R^2/\Gamma$, where $\Gamma$ is a discrete group of isometries acting properly  on $\R^2$,  and the only possible oriented cases are the described ones in the first item of the theorem.
\vspace{.3\baselineskip }\\
{\sc Case II}: $K(\Lambda_a)$ does  not vanishes identically. In this case, we can consider on $\Sigma$ the conformal Riemann structure induced by $\Lambda_a$ and using   Huber's results, see \cite[Theorem 10, Theorem 12, Theorem 13]{Hu},  $\Sigma$ must be conformally either a sphere or a plane.
But if $\Sigma$ is a sphere, from classical Gauss-Bonnet's theorem, \eqref{num} and \eqref{cnm}, we have
$$ 4\pi = \int_\Sigma K(\Lambda_a) dA_{\Lambda_a} = -\int_\Sigma K(I) dA_I =-4\pi, $$
which gives a contradiction.

When $\Sigma$ is not compact, it must be homeomorphic to a plane and from  mentioned Huber's results, \eqref{num} and \eqref{cnm},  we also have  the following  inequality:
$$ \int_\Sigma |K(I)| dA_I  = \int_\Sigma K(\Lambda_a)dA_{\Lambda_a} \leq 2\pi, $$
which concludes the proof.
\end{proof}
From  the proof  of Theorem \ref{at1}, we observe that if $\Sigma$ has a compact boundary we  can also apply  Hubber's results to get that $(\Sigma,I)$ has  finite total curvature. Actually,  we can  easily check   the following result holds:
\begin{teo}\label{at2} Let $\Sigma$ be a surface and $C\subset \Sigma$ a compact subset. Assume $(I,II)$ is a Codazzi pair on $\Sigma\setminus C$ which principal curvatures are strictly separated. If  $I$ is a complete   metric with non positive  Gauss curvature on $\Sigma\setminus C$,  then $(\Sigma\setminus C,I)$ has finite total curvature. In particular,  $\Sigma$ is of parabolic type and has finite topology.
\end{teo}
\subsection{Applications to non Euclidean space forms}
In this subsection we apply the above Theorems \ref{at1} and \ref{at2} to obtain Efimov and Milnor's type results in the hyperbolic space, $\H^3$,  of sectional  curvature $-1$ and in the sphere, $\S^3$, of sectional curvature $1$.

Because Codazzi pairs' theory  appear in the study of surfaces in other target spaces,  analogous results could be given in many others contexts,  for instance spacelike surfaces in the 3-dimensional Lorentzian space form or surfaces in an n-dimensional (semi-Riemannian) space form with  a parallel unit normal vector field $\xi$.

As a first consequence of Theorem \ref{at1} we have,
\begin{cor}\label{ac1}
Let $\psi:\Sigma \longrightarrow \H^3$  be an immersion with  Gauss curvature,  $K\leq -1$, and one of its principal  curvature functions $k_i$ satisfying, 
$$ k_i^2 \geq \epsilon^2 >0,   \qquad \text{ for some constant } \epsilon >0.$$
Then $\psi$ is not a complete immersion.
\end{cor}
\begin{proof}
Up to a change of orientation we can assume $k_2\geq \epsilon>0$. Then using the  Gauss equation of the immersion,
$$ k_1 k_2 = K + 1\leq 0,$$
and we have that the principal curvatures of $\psi$ satisfy $ k_1\leq 0< \epsilon \leq k_2$. 
Thus, the pair $(I,II)$ formed by the first and second fundamental forms of  $\psi$ is a Codazzi pair on $\Sigma$ which   principal curvatures are strictly separated.  If  we assume that $I$ is  complete, then by applying   Theorem \ref{at1},  we  deduce  that $\Sigma$ is homeomorphic to a plane and its area is estimated as follows
$$A(\Sigma) \leq \int_\Sigma |K| dA \leq 2\pi,$$
which contradicts the well-known property  that ``{\sl any simply connected complete Riemannian surface of non positive Gauss curvature  has  infinite area}''. 
\end{proof}
\begin{cor}\label{ac2}
Let $\psi:\Sigma \longrightarrow \S^3$  be an immersion with  Gauss curvature,  $K\leq const < 0$, and one of its principal  curvature functions $k_i$ satisfying, 
$$ k_i^2 \geq \epsilon^2 >0,   \qquad \text{ for some constant } \epsilon >0.$$
Then $\psi$ is not a complete immersion.\end{cor}
\begin{proof}
As above, we can assume $k_2\geq \epsilon>0$. Then, using the  Gauss equation of the immersion,
$$ k_1 k_2 = K - 1 < -1$$
and  the principal curvatures of $\psi$ satisfy the following relation $ k_1\leq 0< \epsilon \leq k_2$. 
Now, the proof follows by applying Theorem \ref{at1} as in the above corollary.
\end{proof}
\begin{remark} {\rm A direct consequence of Theorem \ref{at1} is the non existence of  complete immersed surfaces in $\H^3$ with Gauss curvature $K\leq const <0$  and strictly separated principal curvatures.} 
\end{remark}
\begin{remark} {\rm By changing the hypothesis $K\leq  const <0$ by $K\leq 0$ and $\int_\Sigma |K| dA_I > 2\pi$ in Corollary \ref{ac2} we have the same conclusion.} 
\end{remark}
As straightforward  consequences of Theorem \ref{at2},  we also have,
\begin{cor}\label{c4}
 Let $\psi:\Sigma \longrightarrow \H^3 $ be a complete immersion 
 and $C\subset \Sigma$ a compact subset. Assume that  on $\Sigma\setminus C$ the Gauss curvature of $\psi$ verifies $K\leq -1$ and one of its principal  curvature functions $k_i$ satisfies, 
$$ k_i^2 \geq \epsilon^2 >0,   \qquad \text{ for some constant } \epsilon >0.$$ 
Then $\psi$ has finite area, $\Sigma$ is parabolic and has finite topology.
 \end{cor}
\begin{cor}\label{c5}
 Let $\psi:\Sigma \longrightarrow \S^3 $ be a complete immersion 
 and $C\subset \Sigma$ a compact subset. Assume that  on $\Sigma\setminus C$ the Gauss curvature of $\psi$ verifies $K\leq const < 0$ and one of its principal  curvature functions $k_i$ satisfies, 
$$ k_i^2 \geq \epsilon^2 >0,   \qquad \text{ for some constant } \epsilon >0.$$ 
Then $\psi$ has finite area, $\Sigma$ is parabolic and has finite topology.
\end{cor}
\begin{remark} {\rm One can apply the above corollaries to prove that in $\H^3$ (respectively, $\S^3$) any complete end with Gauss curvature $K\leq -1$ (respectively, $K\leq const < 0$) and one of its principal curvatures $k_i$ satisfying 
$$ k_i^2 \geq \epsilon^2 >0,   \qquad \text{ for some constant } \epsilon >0$$ 
has finite area and finite total curvature.}
\end{remark}

\end{document}